\documentclass[11pt]{article}

\usepackage{amsmath,amssymb,amsfonts,textcomp,amsthm,xifthen,psfrag,graphicx,color,MnSymbol}
\usepackage[T1]{fontenc}
\usepackage{hyperref}

\newtheorem{theorem}{Theorem}
\newtheorem{lemma}[theorem]{Lemma}

\newcommand{\<}{\langle{}}
\renewcommand{\>}{\rangle}
\newcommand{\R}{\ensuremath{\mathbb{R}}}

\newcommand{\brac}[2]{\bigl[#1\,#2\bigr]}
\newcommand{\sbrac}[2]{[#1\,#2]}

\newcommand{\ip}[2]{\llangle#1\hspace*{.5mm},#2\rrangle}

\newcommand{\vdual}[2]{(#1\hspace*{.5mm},#2)}

\newcommand{\tq}{{\hat q}}
\newcommand{\tU}{{\hat U}}
\newcommand{\HH}{\mathbb{H}^2}

\newcommand{\du}{{\delta\!u}}
\newcommand{\dM}{{\delta\!M}}
\newcommand{\dtq}{{\delta\!\tq}}
\newcommand{\dv}{{\delta\!v}}
\newcommand{\dz}{{\delta\!z}}
\newcommand{\dW}{{\delta\!W}}

\newcommand{\cT}{\mathcal{T}}
\newcommand{\ttt}{{\rm T}_t}

\title{A locking-free DPG scheme for Timoshenko beams
\thanks{This research has been supported by CONICYT-Chile through
        Fondecyt projects 1190009, 11170050, 3190359}} 
\author{
Thomas F\"{u}hrer$^{\dagger}$, Carlos Garc\'{i}a Vera$^{\dagger}$,
Norbert Heuer\thanks{Facultad de Matem\'{a}ticas, Pontificia 
Universidad Cat\'olica de Chile, Avenida Vicu\~na Mackenna 4860, Santiago, 
Chile, email: {\tt \{tofuhrer,cgarciv,nheuer\}@mat.uc.cl}}
}

\date{\today}

\begin{document}
\maketitle

\begin{abstract}
We develop a discontinuous Petrov--Galerkin scheme with optimal test functions (DPG method)
for the Timoshenko beam bending model with various boundary conditions,
combining clamped, supported, and free ends. Our scheme approximates
the transverse deflection and bending moment. It converges quasi-optimally
in $L_2$ and is locking free. In particular, it behaves well (converges quasi-optimally)
in the limit case of the Euler--Bernoulli model.
Several numerical results illustrate the performance of our method.


\end{abstract}

%%%%%%%%%%%%%%%%%%%%%%%%%%%%%%%%%%%%%%%%%%%%%%%%%%%%%%%%%%%%%%%%%%%%%%%%%%%%%%%%%
\section{Introduction}

Thin structures like beams, plates, and shells form an important area of research in solid mechanics.
The Reissner--Mindlin model is among the most widely used for the analysis of plate bending.
The corresponding one-dimensional case is the Timoshenko model for beam bending.
Numerical schemes for these models are tricky to design due to the presence of the thickness
parameter, $t$, which induces a singular perturbation when $t\to 0$ in the case of plates.
Not carefully designed approximation schemes suffer from locking.
Mathematically, the limit cases (setting $t=0$ in the scaled models) correspond to the
Kirchhoff--Love and Euler--Bernoulli models in the case of plates and beams, respectively.

There is extensive literature on the numerical analysis of these models generally, though with fewer
results from the mathematics community on the Kirchhoff--Love model, which can suffer from a lack
of regularity. We do not discuss the many contributions that exist but mention some
mathematically focused paper, on the discontinuous Galerkin scheme for Euler--Bernoulli beams from
Baccouch \cite{Baccouch_14_LDG}, and on locking-free $hp$ finite element approaches
for Timoshenko beams from Li and Celiker \emph{et al.} \cite{Li_90_DTB,CelikerCS_06_LFO}.
More recently, Lepe \emph{et al.} presented a locking-free mixed finite element scheme for
the Timoshenko model \cite{LepeMR_14_LFF}.

Here we continue our study of discontinuous Petrov--Galerkin schemes with optimal test
functions (DPG method) for singularly perturbed problems. The underlying idea consists in
using product (``broken'') test spaces and optimal test functions to automatically
satisfy discrete inf--sup properties of Galerkin schemes for any well-posed variational
formulation, see the early works of Demkowicz and Gopalakrishnan, e.g., \cite{DemkowiczG_11_ADM}.
In order to obtain robust (or locking-free) approximations it is critical
to select appropriate test norms \cite{DemkowiczH_13_RDM,ChanHBTD_14_RDM},
possibly in combination with specifically designed variational formulations
\cite{HeuerK_17_RDM,FuehrerHS_UFR}, or adaptively improved test functions \cite{DemkowiczFHT_DAP}.

In this paper we develop a locking-free DPG method for Timoshenko beams that also works in
the limit case of thickness zero (in a scaled version), the Euler--Bernoulli model.
We use our knowledge of variational formulations for fourth-order problems that we have
obtained from our work on the Kirchhoff--Love plate bending model
\cite{FuehrerHN_19_UFK,FuehrerH_19_FDD,FuehrerHH_TOB}, and on the Reissner--Mindlin plate
model \cite{FuehrerHS_UFR}. Specifically, we follow \cite{FuehrerHS_UFR} where we developed
an ultraweak formulation based on the deflection and bending moment, and included the gradient
of the deflection as an independent unknown. The Reissner--Mindlin model has some
critical regularity issues, with weaker deflection and stronger bending moment compared to
the Kirchhoff--Love situation. This makes the analysis interesting.
The techniques presented in this paper for the Timoshenko beam model are based on
those from \cite{FuehrerHS_UFR}. But, instead of closely following the same paths,
we simplify procedures and shorten proofs since in the one-dimensional situation
regularity properties are much simpler. For instance, considering $L_2$ regular distributed
forces, both the deflection and bending moment variables are $H^2$-regular.
Therefore, we can avoid using the additional gradient variable
without complicating the theoretical analysis. Furthermore,
the sophisticated trace operators from \cite{FuehrerHS_UFR} dramatically simplify.
Due to the $H^2$-regularities we are able to use some techniques from
\cite{FuehrerHH_TOB} where we studied the bi-Laplacian in higher space dimensions.

Finally we note that Niemi \emph{et al.} \cite{NiemiBD_11_DPG} have used DPG techniques for beams before.
Though they only consider a cantilever with tip load, without distributed load or different
boundary conditions, and assume the beam thickness to be fixed.

The remainder of this paper is organized as follows.
In Section~\ref{sec_model} we present the model problem, develop an ultraweak variational
formulation, and state its well-posedness (Theorem~\ref{thm_VF}).
The DPG scheme is briefly presented in Section~\ref{sec_DPG}, and Theorem~\ref{thm_DPG}
states its robust quasi-optimal convergence. Section~\ref{sec_proof} gives a proof
of the well-posedness of our variational formulation, split into several lemmas.
Various numerical experiments for different thickness parameters and boundary conditions
are presented in Section~\ref{sec_num}. They confirm that our scheme is locking free.

Throughout this paper, $a \lesssim b$ means that $a\leq cb$ with a generic constant $c>0$ that 
is independent of the thickness parameter $t$ and the underlying mesh. Similarly, we use the 
notation $a \gtrsim b$ and $a \simeq b$. 

%%%%%%%%%%%%%%%%%%%%%%%%%%%%%%%%%%%%%%%%%%%%%%%%%%%%%%%%%%%%%%%%%%%%%%%%%%%%%%%
\section{Model problem and variational formulation} \label{sec_model}

We consider the scaled dimensionless stationary Timoshenko model for beam bending, formulated as
in \cite{NiemiBD_11_DPG} (though with different sign for the bending moment):

\[
   t^2 Q = \theta (u'-\psi),\quad -M = \psi',\quad -Q' = p,\quad M'-kQ = 0
   \qquad\text{in}\quad\ I:=(0,1).
\]
Here, $u$, $\psi$, $Q$, $M$ are, respectively, the transverse deflection, the rotation of the beam's 
cross section, the shear force, and the bending moment. The beam has length $1$ and thickness $t$.
We assume that $t\in [0,1]$.
Furthermore, $p$ is the distributed force and $k$, $\theta$ are constants.
As in \cite{FuehrerHN_19_UFK,FuehrerHS_UFR} for plate problems, we develop
a scheme that approximates the deflection and bending moment variables.
To this end we replace $Q$ by using the relation $M'-kQ = 0$ and eliminate $\psi$, obtaining
\[
   -M'' = f := kp, \qquad t^2 M'' = k\theta(u''+M).
\]
For simplicity we set $k\theta=1$. This parameter is not critical. The strong form of our model
problem then is
\begin{equation} \label{prob}
   -M'' = f,\quad M - t^2 M'' + u'' = 0 \qquad\text{in}\quad I.
\end{equation}
We note that as expected, setting $t = 0$, the Timoshenko model reduces to the Euler--Bernoulli model.
In the following we assume that $f\in L_2(I)$. Then, $u,M\in H^2(I)$.

It remains to specify boundary conditions. We consider all the physically relevant combinations
of clamped end (deflection and rotation are given), supported end (deflection and bending moment
are given), and free end (bending moment and shear force are given), for simplicity of presentation
with homogeneous data only. We note that it is straightforward to consider
non-homogeneous boundary data since all the relevant traces are present in our formulation.
Using the relation $\psi=u'-t^2 M'$ for the rotation, the boundary conditions are

\begin{equation}\label{BC}
\begin{array}{llllllllll}
   \text{clamped-clamped:}
   &u(0)=0,\quad &u'(0)=t^2M'(0),\, &u(1)=0,\, &u'(1)=t^2M'(1), \\  
   \text{clamped-supported:}
   &u(0)=0,\, &u'(0)=t^2M'(0),\, &u(1)=0,\, &M(1)=0,\\ 
   \text{clamped-free:}
   &u(0)=0,\, &u'(0)=t^2M'(0),\, &M(1)=0,\, &M'(1)=0,\\ 
   \text{supported-supported:}
   &u(0)=0,\, &M(0)   =0,\, &u(1)=0,\, &M(1)=0.
\end{array}
\end{equation}
The corresponding solution spaces are
\begin{align*}
   &\HH_{cc}(t) := \{(u,M)\in H^2(I)\times H^2(I);\; u(0)=u'(0)-t^2 M'(0)=u(1)=u'(1)-t^2 M'(1)=0\},\\
   &\HH_{cs}(t) := \{(u,M)\in H^2(I)\times H^2(I);\; u(0)=u'(0)-t^2 M'(0)=u(1)=M(1)=0\},\\
   &\HH_{cf}(t) := \{(u,M)\in H^2(I)\times H^2(I);\; u(0)=u'(0)-t^2 M'(0)=M(1)=M'(1)=0\},\\
   &\HH_{ss}(t) := \{(u,M)\in H^2(I)\times H^2(I);\; u(0)=M(0)=u(1)=M(1)=0\}.
\end{align*}
Of course, $\HH_{ss}(t)$ is independent of $t$.

Now we derive an ultraweak formulation of our Timoshenko beam problem.
For a positive integer $n$ and nodes $0=x_0 < x_1 < x_2 < \ldots <x_n=1$, let us consider the partition 
$\cT=\{I_j=(x_{j-1},x_j);\; j=1,\ldots,n\}$ of $I$. Below, we denote $h_j:=x_j-x_{j-1}$
($j=1,\ldots,n$). Testing the equations from \eqref{prob} respectively with
\[
   z,W\in H^2(\cT):=\{v\in L_2(I);\; v_j:=v|_{I_j}\in H^2(I_j),\ j=1,\ldots,n\},
\]
integration by parts gives
\begin{align}\label{VFa}
  &\vdual{u}{W''}_\cT + \vdual{M}{W+z''-t^2W''}_\cT
  \nonumber\\
  &+
  \sum_{j=1}^n \Bigl(-\brac{u}{W'}_j + \brac{M'}{z}_j
                     - \brac{M}{(z'-t^2W')}_j + \brac{(u'-t^2 M')}{W}_j\Bigr)
  =  -\vdual{f}{z}.
\end{align}
Here, $\vdual{\cdot}{\cdot}$ denotes the $L_2(I)$-inner product with norm $\|\cdot\|$, and
$\vdual{\cdot}{\cdot}_\cT$ indicates the $L_2$-inner product that is taken piecewise on $\cT$.
The $L_2(T)$-norm ($T\in\cT$) will be denoted by $\|\cdot\|_T$.
Furthermore, $\sbrac{\,\cdot}{}_j$ are the boundary terms from the integration-by-parts formula on
$I_j$. That is, e.g., $\sbrac{u}{W'}_j=u(x_j)W'(x_j)-u(x_{j-1})W'(x_{j-1})$ where the point evaluations
are taken from $u$ and $W$ restricted to $I_j$ ($j=1,\ldots,n$).

Introducing the following maps for the point evaluations,
\begin{align*}
   H^2(I_j)\ni z \mapsto &\gamma_j(z):= (z(x_{j-1}),z'(x_{j-1}),z(x_j),z'(x_j)),\quad j=1,\ldots,n,\\
   H^2(\cT)\ni z \mapsto &\gamma_h(z):= \Bigl(\gamma_j(z_j))\Bigr)_{j=1}^n
\end{align*}
(note that $z_j=z|_{I_j}$), the point evaluations from \eqref{VFa} are abbreviated as
\[
   \<\gamma_h(u,M),(z,W)\>_t :=
  \sum_{j=1}^n \Bigl(-\brac{u}{W'}_j + \brac{M'}{z}_j
                     - \brac{M}{(z'-t^2W')}_j + \brac{(u'-t^2 M')}{W}_j\Bigr).
\]
Here we abuse the notation and write $\gamma_h(u,M)=(\gamma_h(u),\gamma_h(M))$.
Of course, for $u\in H^2(I)$, $\gamma_h(u)$ gives rise to $2n+2$ independent real numbers
($\gamma_h$ induces a continuous surjective mapping from $H^2(I)$ to $\R^{2n+2}$),
and we note that
\[
   \<\gamma_h(u,M),(z,W)\>_t = - \<\gamma_h(z,W),(u,M)\>_t \quad\forall u,M,z,W\in H^2(\cT).
\]
For all of the boundary conditions from \eqref{BC}, $\gamma_h(u,M)$ allows for $4n$ independent scalar
unknowns. Specifically, depending on the type of boundary condition, we denote
\[
   \tU_a(t) := \gamma_h(\HH_a(t)),\quad a\in\{cc, cs, cf, ss\}.
\]
As noted before, the dimension of any of these spaces is $4n$, and $\tU_{ss}(t)$
is independent of $t$. We also need the image
\[
   \tU:=\gamma_h(H^2(I)\times H^2(I)) = (\gamma_h(H^2(I)),\gamma_h(H^2(I)) = \R^{4n+4}.
\]
A duality between any of these spaces with $H^2(\cT)\times H^2(\cT)$ is also denoted by
$\<\cdot,\cdot\>_t$, defined as
\[
   \<\tq,(z,W)\>_t := \<\gamma_h(u,M),(z,W)\>_t\quad\text{for}\quad
   u,M\in H^2(I)\ \text{with}\ \gamma_h(u,M)=\tq.
\]
Now, to present the ultraweak variational formulation of \eqref{prob} with one of the
boundary conditions \eqref{BC}, we denote the solution space as
\[
   U_a(t) := L_2(I)\times L_2(I)\times \tU_{a}(t)
   \quad\text{with}\quad a\in\{cc, cs, cf, ss\}\quad\text{as needed},
\]
and the test space as
\[
   V := H^2(\cT)\times H^2(\cT).
\]
Defining the norms (squared)
\begin{align} \label{discrete}
   \|\gamma_j(u)\|_j^2
   &:=
   \frac{h_j}{420}
   \Bigl[     156 \bigl(u(x_{j-1})^2 + u(x_j)^2\bigr) + 4 h_j^2 \bigl(u'(x_{j-1})^2 + u'(x_j)^2\bigr)
   \nonumber\\
   & \qquad + 108 u(x_{j-1}) u(x_j) + 44 h_j \bigl(u(x_{j-1})u'(x_{j-1}) - u(x_j) u'(x_j)\bigr)
   \nonumber\\
   & \qquad - 6 h_j^2 u'(x_{j-1}) u'(x_j) + 26 h_j \bigl(u(x_j)u'(x_{j-1}) - u(x_{j-1})u'(x_j)\bigr)
   \Bigr]
   \nonumber\\
   &+
   \frac{2}{h_j^3}\Bigl[6 \bigl(u(x_{j-1})^2 -2 u(x_{j-1}) u(x_j) + u(x_j)^2\bigr)
                         + 2 h_j^2 \bigl(u'(x_{j-1})^2 + u'(x_{j-1}) u'(x_j) + u'(x_j)^2\bigr)
   \nonumber\\
   &\qquad + 6 h_j \bigl(u(x_{j-1})u'(x_{j-1}) - u(x_j) u'(x_j)
                         + u(x_{j-1})u'(x_j) - u(x_j)u'(x_{j-1})\bigr)
                  \Bigr],
\end{align}
\begin{align*}
   &\|\gamma_h(u,M)\|_h^2
   := \sum_{j=1}^n \Bigl(\|\gamma_j(u_j)\|_j^2 + \|\gamma_j(M_j)\|_j^2\Bigr)
   && (u,M\in H^2(I)),\\
   &\|z\|_{2,T}^2 := \|z\|_T^2 + \|z''\|_T^2 && (z\in H^2(T),\ T\in\cT),\\
   &\|z\|_2^2 := \|z\|^2 + \|z''\|^2 && (z\in H^2(I)),\\
   &\|z\|_{2,\cT}^2 := \|z\|^2 + \vdual{z''}{z''}_\cT && (z\in H^2(\cT)),
\end{align*}
these spaces are normed as follows,
\[
   \|(u,M,\tq)\|_U^2 := \|u\|^2 + \|M\|^2 + \|\tq\|_h^2
   \quad\text{and}\quad
   \|(z,W)\|_V^2 := \|z\|_{2,\cT}^2 + \|W\|_{2,\cT}^2
\]
for $(u,M,\tq)\in U_a(t)$ ($a\in \{cc, cs, cf, ss\}$) and $z,W\in H^2(\cT)$.
Finally, defining the functional $L_f(z,W):=-\vdual{f}{z}$ and bilinear form
\[
   b_t((u,M,\tq),(z,W))
   :=
   \vdual{u}{W''}_\cT + \vdual{M}{W+z''-t^2W''}_\cT + \<\tq,(z,W)\>_t,
\]
the variational formulation is:
\begin{equation} \label{VF}
   (u,M,\tq)\in U_a(t):\quad
   b_t((u,M,\tq),(z,W)) = L_f(z,W)\quad\forall (z,W)\in V.
\end{equation}

Our first main result is the well-posedness of \eqref{VF}.

\begin{theorem} \label{thm_VF}
Given $f\in L_2(I)$, $a\in\{cc,cs,cf,ss\}$, and $t\in [0,1]$, there exists a unique solution
$(u,M,\tq)\in U_a(t)$ of \eqref{VF}. It satisfies
\[
   \|(u,M,\tq)\|_{U(t)} \lesssim \|f\|
\]
with a hidden constant that is independent of $\cT$, $f$, and $t$. Furthermore,
$(u,M)$ solves \eqref{prob}, and $\tq=\gamma_h(u,M)$.
\end{theorem}

%%%%%%%%%%%%%%%%%%%%%%%%%%%%%%%%%%%%%%%%%%%%%%%%%%%%%%%%%%%%%%%%%%%%%%%%%%%%%%%%%
\section{The DPG method} \label{sec_DPG}

Our DPG method is a Petrov--Galerkin scheme based on formulation \eqref{VF}.
We consider a finite-dimensional subspace $U_{a,h}(t)\subset U_a(t)$
and select test functions $(z,W)\in \ttt(U_{a,h}(t))$ where $\ttt$ is the trial-to-test operator
$\ttt:\;U_a(t)\to V$ defined by
\[
   \ip{\ttt(u,M,\tq)}{(z,W)}_V = b_t((u,M,\tq),(z,W))
   \quad \forall (z,W) \in V.
\]
Here, $\ip{\cdot}{\cdot}_V$ is the inner product in $V$,
\[
   \ip{(z,W)}{(\dz,\dW)}_V = \vdual{z}{\dz} + \vdual{z''}{\dz''} + \vdual{W}{\dW} + \vdual{W''}{\dW''}.
\]
Our DPG approximation $(u_h,M_h,\tq_h)\in U_{a,h}(t)$ is defined by
\begin{equation} \label{DPG}
   b_t((u_h,M_h,\tq_h),\ttt(\du,\dM,\dtq)) = L(\ttt(\du,\dM,\dtq))
   \quad \forall (\du,\dM,\dtq)\in U_{a,h}(t).
\end{equation}
Since the DPG-approximation minimizes the residual in the $V$-norm, cf.~\cite{DemkowiczG_11_ADM},
and since the $U$-norm is uniformly equivalent to the dual norm of $V$, cf.~\eqref{infsup} below,
our approximation is quasi-optimal in the $U$-norm. Let us formulate this result.

\begin{theorem} \label{thm_DPG}
Let $a\in\{cc,cs,cf,ss\}$, $f\in L_2(I)$ and $t\in [0,1]$ be given.
There exists a unique solution $(u_h,M_h,\tq_h)\in U_{a,h}(t)$ to \eqref{DPG}. It satisfies
\[
   \|u-u_h\| + \|M-M_h\| + \|\tq-\tq_h\|_h
   \lesssim
   \inf\{\|u-v_h\| + \|M-Q_h\|;\; (v_h,Q_h)\in U_{a,h}(t)\}
\]
with a hidden constant that is independent of $f,\cT,U_{a,h}(t)$, and $t$.
\end{theorem}

%%%%%%%%%%%%%%%%%%%%%%%%%%%%%%%%%%%%%%%%%%%%%%%%%%%%%%%%%%%%%%%%%%%%%%%%%%%%%%%%%
\section{Proof of Theorem~\ref{thm_VF}} \label{sec_proof}

We split the proof of Theorem~\ref{thm_VF} into several parts.
This is standard procedure, cf., e.g., \cite{CarstensenDG_16_BSF,FuehrerHH_TOB,FuehrerHS_UFR}.
Specifically, we closely follow the presentation from \cite{FuehrerHS_UFR}.
The steps consist in,
(a) characterizing the norms for the trace space $\tU$ which is finite-dimensional
(Lemma~\ref{la_trace}),
(b) checking that test functions become continuous when annihilated by trace elements
(Lemma~\ref{la_jump}),
(c) proving stability of the adjoint problem (Lemma~\ref{la_adj}), and
(d) checking injectivity of the operator that is adjoint to problem \eqref{prob}
(Lemma~\ref{la_inj}).

\begin{lemma} \label{la_trace}
\[
   \sup_{(z,W) \in V\setminus\{0\}}
   \frac{\<\tq,(z,W)\>_t}{\|(z,W)\|_V} \simeq \|\tq\|_{h,t}
   \quad \forall\tq \in \tU,\ t\in [0,1].
\]
\end{lemma}

\begin{proof}
{\bf Step 1.}
We start by making three observations. First,
\begin{equation} \label{1a}
   (z,W) \mapsto (\tilde z,\widetilde W):=(z+t^2W,W)
\end{equation}
is an automorphism in $V=H^2(\cT)\times H^2(\cT)$ that is uniformly bounded with respect to $t$
in both directions. Second,
\begin{align} \label{1b}
   \<\gamma_h(u,M),(\tilde z,\widetilde W)\>_t 
   &=
   \sum_{j=1}^n \Bigl(-\brac{u}{W'}_j + \brac{M'}{z}_j - \brac{M}{z'}_j + \brac{u'}{W}_j\Bigr)
   \nonumber\\
   &=
   -\vdual{u}{W''}_\cT + \vdual{u''}{W} - \vdual{M}{z''}_\cT + \vdual{M''}{z}
\end{align}
holds for $u,M\in H^2(I)$ and $z,W\in H^2(\cT)$,
and third, as in \cite[Lemma 1]{FuehrerHH_TOB} for the Laplacian, one proves that
\begin{multline*}
   \sup_{\dv\in H^2(T)\setminus\{0\}}
   \frac {\vdual{v}{\dv''}_T-\vdual{v''}\dv_T}{\|\dv\|_{2,T}}\\
   =
   \inf\{\|w\|_{2,T}; w\in H^2(T),\ w(x_i)=v(x_i),\ w'(x_i)=v'(x_i),\ i=j-1,j\}
   \quad (T=I_j\in\cT)
\end{multline*}
holds for any $v\in H^2(T)$. Therefore, one deduces that
\begin{multline*}
   \sup_{(z,W) \in V\setminus\{0\}}
   \frac{\<\tq,(z,W)\>_t}{\|(z,W)\|_V}
   \simeq\\
   \inf\{\|v\|_2;\; v\in H^2(I),\ \gamma_h(v)=\tq_1\}
   +
   \inf\{\|v\|_2;\; v\in H^2(I),\ \gamma_h(v)=\tq_2\}
\end{multline*}
holds where $\tq=(\tq_1,\tq_2)$.
Noting that the norms on the right-hand side localize (the minima can be calculated
element-wise), it remains to show the local relation
\begin{equation} \label{2}
   \inf\{\|v\|_{2,T};\; v\in H^2(T),\ \gamma_j(v)=\gamma_j(u)\}
   \simeq
   \|\gamma_j(u)\|_j
   \quad \forall u\in H^2(T),\ T=I_j,\ j=1,\ldots, n.
\end{equation}
{\bf Step 2.} To prove \eqref{2} it
is enough to consider an element $I_h=(0,h)$ of length $h\in (0,1]$. Let $u\in H^2(I_h)$ be
given and set $\tq_u=\gamma(u):=(u(0), u'(0), u(h), u'(h))$. It follows that
\[
   \inf\{\|v\|_{2,I_h};\; v\in H^2(T),\ \gamma(v)=\tq_u\}
   =
   \|w_u\|_{2,I_h}
\]
where $w_u\in H^2(I_h)$ solves
\[
   w_u^{(4)}+w_u= 0\quad\text{in}\ I_h,\quad \gamma(w_u)=\tq_u.
\]
We also define $v_u\in H^2(I_h)$ as the solution to
\begin{equation} \label{v}
   v_u^{(4)}=0\quad\text{in}\ I_h,\quad \gamma(v_u)=\tq_u.
\end{equation}
It follows that $\|w_u\|_{2,I_h}\le \|v_u\|_{2,I_h}$ (since $w_u$ minimizes the $H^2(I_h)$-norm)
and $|v_u|_{2,I_h}:=\|v_u''\|_{I_h}\le |w_u|_{2,I_h}\le \|w_u\|_{2,I_h}$ (since $v_u$ minimizes the
$H^2(I_h)$-seminorm). Proving that $\|v_u\|_{I_h}\lesssim \|w_u\|_{2,I_h}$ then gives the uniform
equivalence $\|v_u\|_{2,I_h}\simeq\|w_u\|_{2,I_h}$. To show this, let $z\in H^2(I_h)$ solve
\[
   z^{(4)} = v_u\quad\text{in}\ I_h,\quad \gamma(z)=0.
\]
Then, integrating by parts, using that $\gamma(w_u)=\gamma(v_u)$
and $\vdual{v_u''}{z''}_{I_h}=0$ by \eqref{v} since $\gamma(z)=0$, we conclude that
\begin{align*}
   \|v_u\|_{I_h}^2 = \vdual{v_u}{z^{(4)}}_{I_h}
   &= \vdual{w_u}{z^{(4)}}_{I_h} - \vdual{w_u''}{z''}_{I_h}\\
   &\le
   \|w_u\|_{2,I_h} \bigl(\|z''\|_{I_h}^2 + \|z^{(4)}\|_{I_h}^2\bigr)^{1/2}
   =
   \|w_u\|_{2,I_h}
   \bigl(\|z''\|_{I_h}^2 + \|v_u\|_{I_h}^2\bigr)^{1/2}.
\end{align*}
In the last step we used the relation $z^{(4)}=v_u$.
To conclude that $\|v_u\|_{I_h}\lesssim \|w_u\|_{2,I_h}$ we are left to show that
$\|z''\|_{I_h}\lesssim \|v_u\|_{I_h}$. This is true by the definition of $z$ and
the Poincar\'e--Friedrichs inequality in $H^2_0(I_h)$,
\[
   \|z''\|_{I_h}^2 = \vdual{v_u}{z}_{I_h} \le \|v_u\|_{I_h} \|z\|_{I_h}
   \lesssim h^2 \|v_u\|_{I_h} \|z''\|_{I_h}.
\]
We have shown that the solution $v_u$ of \eqref{v} satisfies
\[
   \inf\{\|v\|_{2,I_h};\; v\in H^2(T),\ \gamma(v)=\tq_u\}
   \simeq
   \|v_u\|_{2,I_h} \quad\forall u\in H^2(I_h).
\]
{\bf Step 3.} We calculate the $H^2(I_h)$-norm of $v_u$
which is a cubic polynomial. Ordering the boundary values as $(u(0),u(h),u'(0),u'(h))$ and
using standard Hermite polynomials,
the $L_2(I_h)$-norm and $H^2(I_h)$-seminorm are induced by the mass and stiffness matrices,
respectively,
\[
   \frac h{420}\begin{pmatrix}
      156 & 54 & 22h & -13h\\
      54 & 156 & 13h & -22h\\
      22h & 13h & 4h^2 & -3h^2\\
      -13h & -22h & -3h^2 & 4h^2
   \end{pmatrix},
   \qquad
   \frac 2{h^3}\begin{pmatrix}
      6 & -6 & 3h & 3h\\
      -6 & 6 & -3h & -3h\\
      3h & -3h & 2h^2 & h^2\\
      3h & -3h & h^2 & 2h^2
   \end{pmatrix}.
\]
They give the weights for the discrete norms, cf.~\eqref{discrete}.
We have thus proved \eqref{2} and the lemma.
\end{proof}

\begin{lemma} \label{la_jump}
Let $(z,W)\in V$, $a\in\{cc, cs, cf, ss\}$, and $t\in [0,1]$ be given. The relation
\[
   (z,W)\in\HH_a(t)
   \quad\Leftrightarrow\quad
   \<\tq,(z,W)\>_t = 0\quad\forall \tq\in\tU_a(t)
\]
holds true.
\end{lemma}

\begin{proof}
The relation $(z,W)\in\HH_a(t)\Rightarrow \<\tq,(z,W)\>_t = 0$ for any $\tq\in \tU_a(t)$
follows from integration by parts. For the other direction we use the transformation
\eqref{1a} and relation \eqref{1b} to conclude that $(z,W)\in H^2(I)\times H^2(I)$.
The boundary conditions are obtained analogously.
\end{proof}

\begin{lemma} \label{la_adj}
Given $a\in\{cc, cs, cf, ss\}$, $t\in[0,1]$, and $g,k\in L_2(I)$,
there exists a unique solution $(z,W)\in\HH_a(t)$ to
\begin{equation} \label{adj}
   W'' = g,\quad W - t^2 W'' + z'' = k\quad \text{in}\ I.
\end{equation}
It satisfies
\[
   \|z\|_2 + \|W\|_2 \lesssim \|g\| + \|k\|
\]
with a hidden constant that is independent of $g,k$ and $t$.
\end{lemma}

\begin{proof}
We start with the case $t\in (0,1]$.
Testing the equations in \eqref{adj} with $\dz\in H^1(I)$ and $\dW\in H^1(I)$, respectively,
and boundary conditions
\begin{align*}
   &\dz(0)=\dz(1)=0        &&\hspace*{-8em} (\text{if}\ a=cc),\\
   &\dz(0)=\dz(1)=\dW(1)=0 &&\hspace*{-8em} (\text{if}\ a=cs),\\
   &\dz(0)=\dW(1)=0        &&\hspace*{-8em} (\text{if}\ a=cf),\\
   &\dz(0)=\dW(0)=\dz(1)=\dW(1)=0        &&\hspace*{-8em} (\text{if}\ a=ss),
\end{align*}
we obtain
\begin{align} \label{VFadj}
   \vdual{W}{\dW} + t^2 \vdual{W'}{\dW'} - \vdual{z'}{\dW'} - \vdual{W'}{\dz'}
   = \vdual{k}{\dW} + \vdual{g}{\dz}.
\end{align}
Let us denote the subspaces of $H^1(I)$ that satisfy the boundary conditions for
$W$ (not for $W'$) and $\dz$ by $H^1_W$ and $H^1_\dz$, respectively.
In all the four cases the second order derivative is a surjective map from $H^1_W$
to the dual of $H^1_\dz$. Therefore, the inf--sup property
\begin{equation} \label{is}
   \sup_{W\in H^1_W\setminus\{0\}} \frac {\vdual{W'}{\dz'}}{(\|W\|^2 + t^2 \|W'\|^2)^{1/2}}
   \ge
   \sup_{W\in H^1_W\setminus\{0\}} \frac {\vdual{W'}{\dz'}}{(\|W\|^2 + \|W'\|^2)^{1/2}}
   \gtrsim
   \bigl(\|\dz\|^2 + \|\dz'\|^2\bigr)^{1/2}
\end{equation}
holds uniformly for $\dz\in H^1_\dz$.
Noting that the conditions for the boundary values of $z$ (not of $z'$) are identical
to those of $\dz$, and respectively those of $W$ and $\dW$,
\eqref{VFadj} is a standard mixed formulation in the space $H^1_W\times H^1_z$
(with $H^1_z:=H^1_\dz$). Inf--sup property \eqref{is} and the
$H^1(I)$-coercivity of the bilinear form $\vdual{W}{\dW} + t^2 \vdual{W'}{\dW'}$
show that \eqref{VFadj} has a unique solution $(z,W)\in H^1_z\times H^1_W$ with bound
\[
   \|W\|^2 + t^2 \|W'\|^2 + \|z\|^2 + \|z'\|^2 \lesssim \|g\|^2 + \|k\|^2.
\]
Here, the hidden constant is independent of $k,g$ and $t\in (0,1]$.

Finally, relations \eqref{adj} give
$\|W''\|=\|g\|$ and $\|z''\|\le\|k\|+\|W\|+\|W''\|\lesssim \|g\|+\|k\|$.
It is also easy to check the remaining boundary conditions for $z$ and $W$.
This concludes the proof in the case of $t\in (0,1]$.

Now we consider $t=0$. Problem \eqref{adj} then reduces to
\[
   z^{(4)} = k''-g \quad\text{in}\ I\qquad (\text{and}\ W:=k-z'')
\]
where the boundary conditions for $z$ are imposed either directly or via conditions on $W=k-z''$.
In all cases, testing the differential equation with $\dz\in H^2(I)$ subject to
\[
   \dz=\dz'=0 \text{ where clamped},\quad
   \dz=0 \text{ where supported},
\]
and without condition on free boundaries, we obtain
\[
   \vdual{z''}{\dz''} = \vdual{k}{\dz''} - \vdual{g}{\dz}.
\]
In all four cases $a\in\{cc, cs, cf, ss\}$, this is a coercive problem
with unique solution $z$ satisfying
\[
   \|z\|_2 \lesssim \|g\| + \|k\|.
\]
The bound for $\|W\|_2$ is obtained through $W''=g$ and $W=k-z''$.
Again, $z$ and $W$ satisfy the required boundary conditions.
\end{proof}

\begin{lemma} \label{la_inj}
Let $a\in\{cc,cs,cf,ss\}$ and $t\in [0,1]$ be given.
If $(z,W)\in V$ satisfies
\[
   b_t((\du,\dM,\dtq),(z,W)) = 0\quad\forall (\du,\dM,\dtq)\in U_a(t)
%   \vdual{u}{W''}_\cT + \vdual{M}{W+z''-t^2W''}_\cT + \<\tq,(z,W)\>_t
\]
then $u=W=0$.
\end{lemma}

\begin{proof}
Lemma~\ref{la_jump} implies that $(z,W)\in\HH_a(t)$. Then, selecting separately $(\du,\dM)=(\du,0)$
and $(\du,\dM)=(0,\dM)$, we obtain $W''=0$ and $W-t^2W''+z''=0$ in $I$. That is,
$(z,W)$ solves \eqref{adj} with $g=k=0$, so that $z=W=0$ by Lemma~\ref{la_adj}.
\end{proof}

\noindent
{\bf Proof of Theorem~\ref{thm_VF}.}
It is enough to prove the well-posedness of \eqref{VF}. The relation of the solution
to \eqref{VF} with problem \eqref{prob} is clear.

The well-posedness of \eqref{VF} is proved in the standard way.
It is clear that the bilinear form $b_t(\cdot,\cdot)$ and functional $L_f$ are bounded in $U(t)$
uniformly with respect to $t$ (the term of the bilinear form comprising the point evaluations
is bounded with the help of Lemma~\ref{la_trace}).
Therefore, it remains to check the two inf--sup conditions for $b_t(\cdot,\cdot)$.
Lemma~\ref{la_inj} implies injectivity,
\[
   \sup_{(\du,\dM,\dtq)\in U_a(t)} b_t((\du,\dM,\dtq),(z,W)) > 0
   \quad\forall (z,W)\in V\setminus\{0\}.
\]
By \cite[Theorem 3.3]{CarstensenDG_16_BSF}, the second inf--sup property,
\begin{equation} \label{infsup}
   \sup_{(z,W)\in V\setminus\{0\}} \frac {b_t((u,M,\tq),(z,W))}{\|(z,W)\|_V}
   \gtrsim
   \|(u,M,\tq)\|_U
   \quad\forall (u,M,\tq)\in U_a(t),
\end{equation}
follows from three results. First, the inf--sup condition
\[
   \sup_{(z,W) \in \HH_a(t)\setminus\{0\}}
   \frac{b_t((u,M,0),(z,W))}{\|(z,W)\|_V}
   \gtrsim \|u\| + \|M\| \quad\forall u,M\in L_2(I)
\]
is needed. It can be proved by an appropriate selection of test functions and using
Lemma~\ref{la_adj}. Second, we need the inf--sup condition
\[
   \sup_{(z,W)\in V\setminus\{0\}}
   \frac{\<\tq,(z,W)\>_t}{\|(z,W)\|_V}
   \gtrsim
   \|\tq\|_h\quad\forall \tq\in \tU_a(t),
\]
which is true by Lemma~\ref{la_trace}. Finally, the relation
\[
   \HH_a(t) = \{(z,W) \in V;\; \<\tq,(z,W)\>_t=0\  \forall \tq\in \tU_a(t)\}
\]
is needed. It holds by Lemma~\ref{la_jump}.
This finishes the proof of Theorem~\ref{thm_VF}.
\qed

%%%%%%%%%%%%%%%%%%%%%%%%%%%%%%%%%%%%%%%%%%%%%%%%%%%%%%%%%%%%%%%%%%%%%%%%%%%%%%%%%%%%%%%%
\section{Numerical results} \label{sec_num}
 
In this section we report on numerical experiments with our DPG scheme.
Throughout we consider problem \eqref{prob} with $f(x)=\sin(\pi x)$, the
clamped-free boundary condition ``$cf$'', cf.~\eqref{BC},
and thickness parameter $t\in\{1,10^{-3},10^{-6},0\}$.

We consider uniform meshes $\cT$.
The approximation space $U_{cf,h}(t)$ uses piecewise polynomials on $\cT$ of degree $p\in\{0,1,2\}$
both for $u$ and $M$. The trial-to-test operator $\ttt$ is approximated by replacing
$V$ with piecewise polynomial spaces on $\cT$ of degree $p+3$ for both components.

Figures~\ref{fig_p0}--\ref{fig_p2} present the results for $p=0,1,2$, respectively.
All the graphs give the errors versus the number of degrees of freedom, along
with a curve $O(h^{p+1})$ scaled to fit the plotted range ($h$ is the mesh size).
The legends are identical in all graphs, except for the curve indicating the convergence order.
Specifically,
``residual'' indicates the (approximated) error of the residual in the $V$-norm:
$\|L-b_t((u_h,M_h,\tq_h),\ttt(\cdot))\|_{V'}$,
``u'' and ``M'' refer to the $L_2$-errors $\|u-u_h\|$ and $\|M-M_h\|$, respectively,
``proj(u)'' and ``proj(M)'' indicate the respective $L_2$-errors of the best approximation,
and ``tr(u)'' resp. ``tr(M)'' refer to the parts of the trace error
$\|\tq-\tq_h\|_h$ that stem from $u$ resp. $M$.

Since $u$ and $M$ are smooth functions and the method is quasi-optimal in the $L_2$-norm
by Theorem~\ref{thm_DPG}, we expect a convergence of order $O(h^{p+1})$.
This rate is confirmed in all the cases, whereas the trace errors converge faster than predicted
(except for the trace of $u$ when $p=0$).
We also note stability issues for large dimensions. This is expected by the large
condition number of the stiffness matrix that behaves like $O(h^{-4})$.
Before reaching large dimensions the results are practically independent of $t$ when
$t=10^{-3},10^{-6},0$. Thus, our scheme is robust with respect to $t$, it is locking free.
We also observe that the errors for $u$ and $M$ are
indistinguishable from their best-approximation errors, until round-off errors appear.
Finally we note that numerical experiments (not reported) show very similar behavior
of the method for the other boundary conditions.

\begin{figure}[htb]
\includegraphics[width=0.5\textwidth]{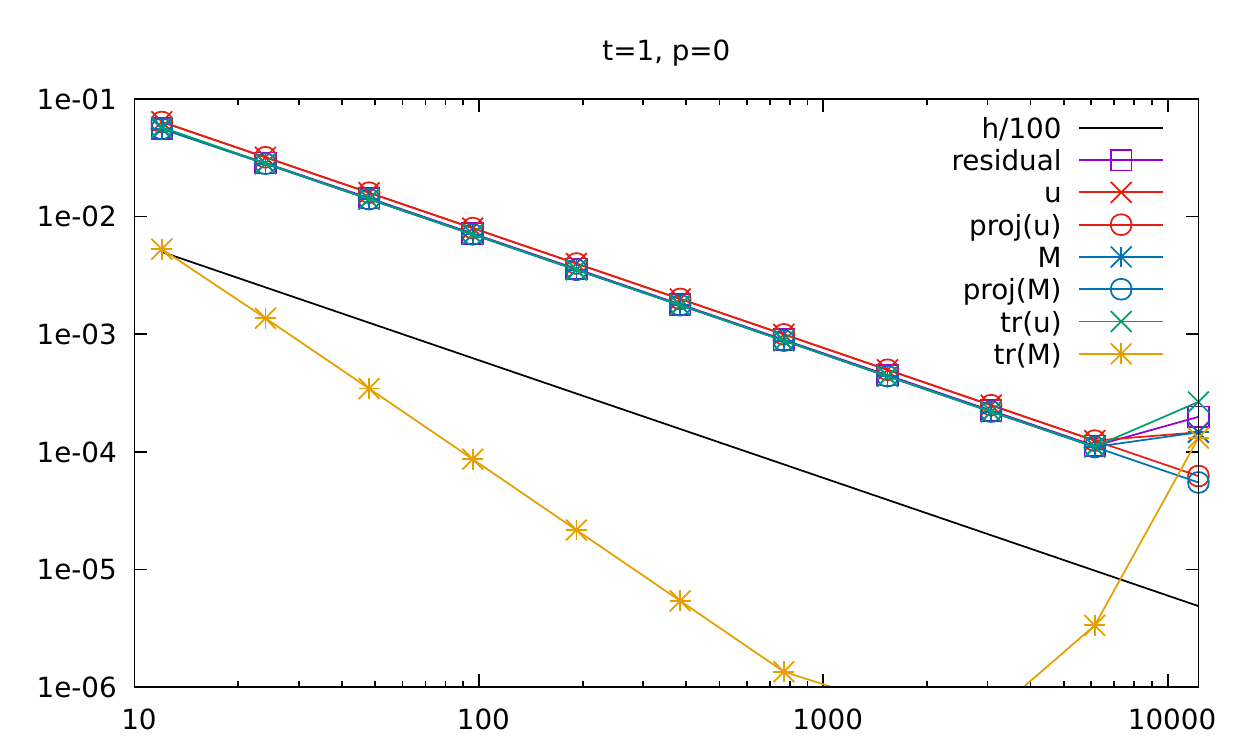}\hspace*{-0.5em}
\includegraphics[width=0.5\textwidth]{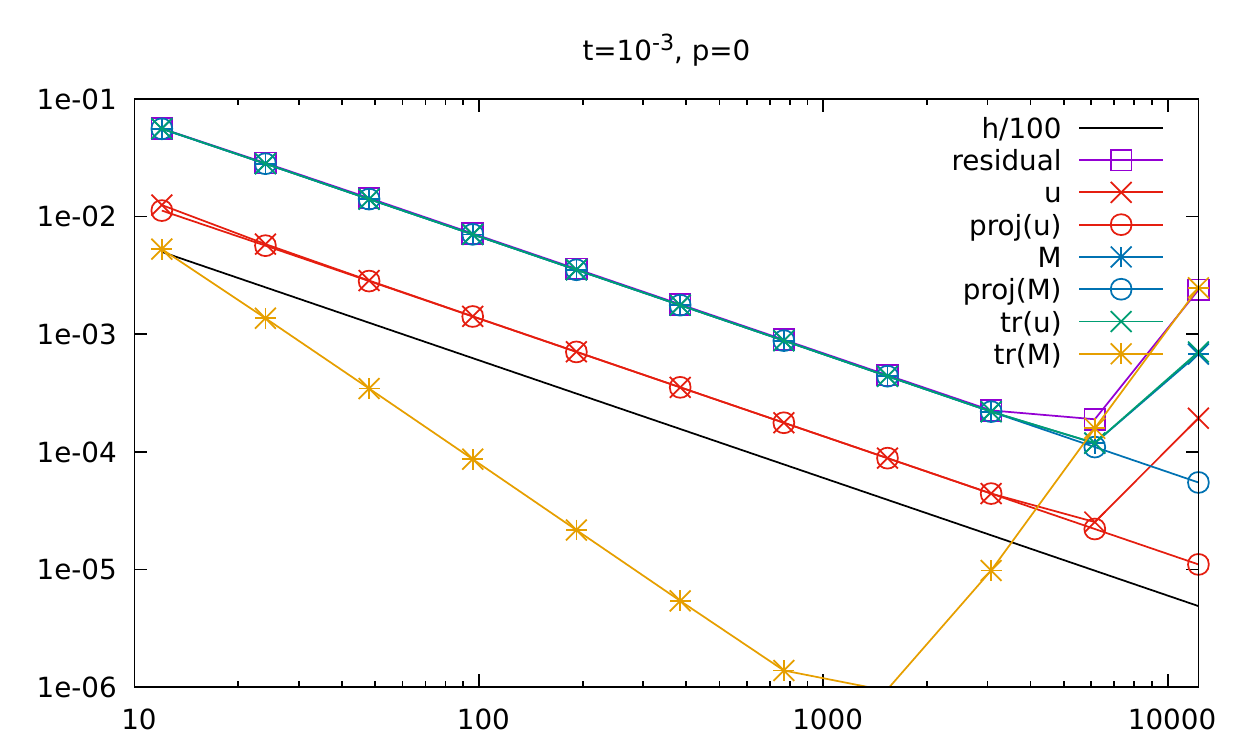}\\
\includegraphics[width=0.5\textwidth]{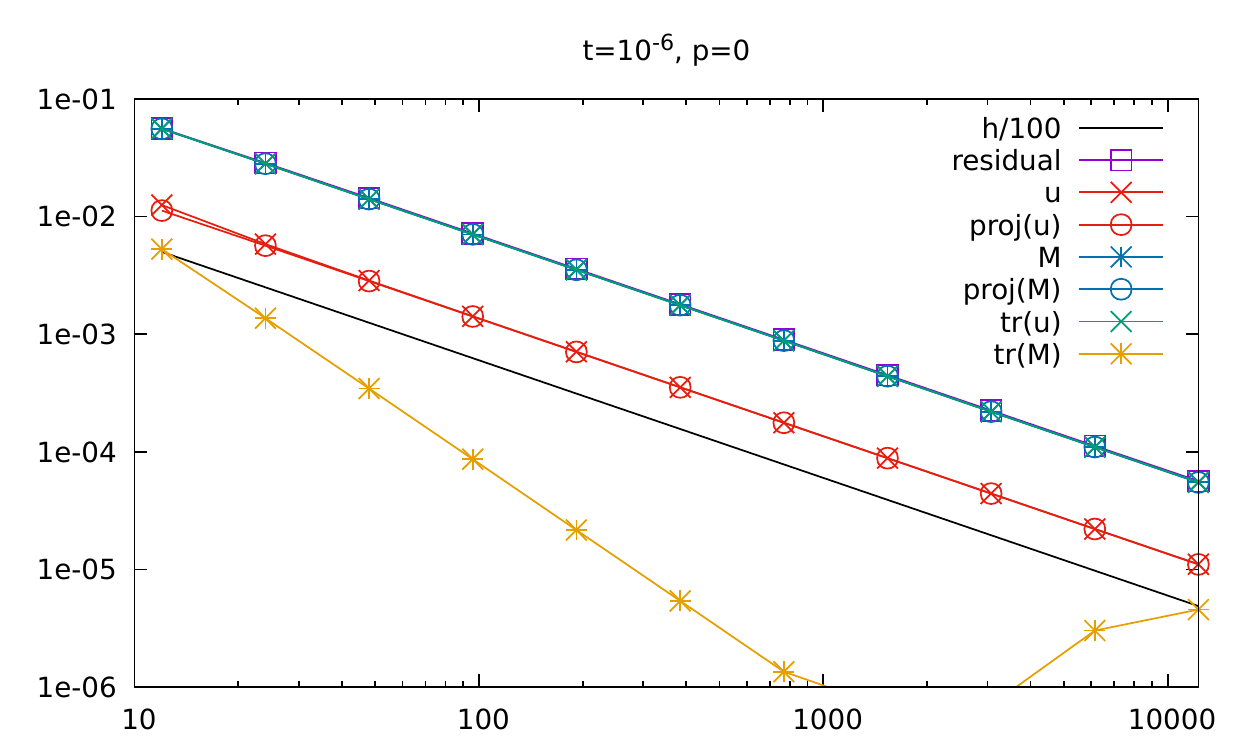}\hspace*{-0.5em}
\includegraphics[width=0.5\textwidth]{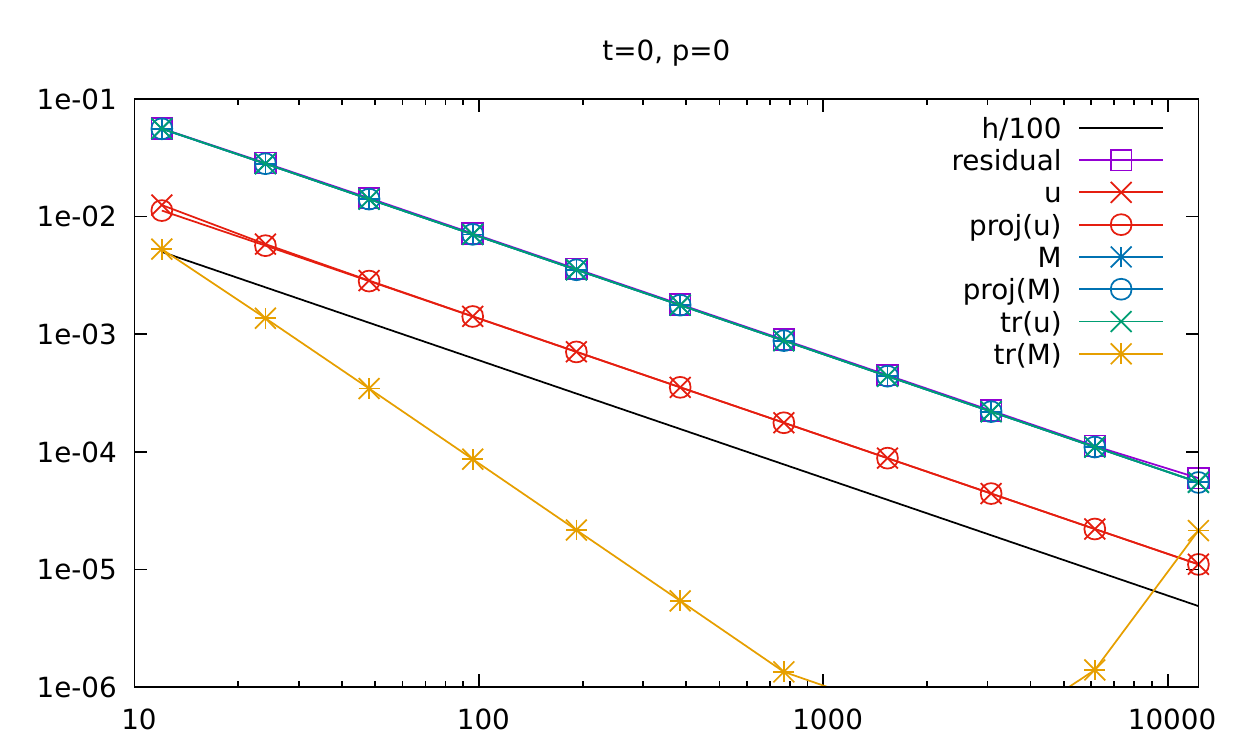}
\caption{Errors in the case of piecewise constant approximations, $t=1,10^{-3},10^{-6},0$.}
\label{fig_p0}
\end{figure}

\begin{figure}[htb]
\includegraphics[width=0.5\textwidth]{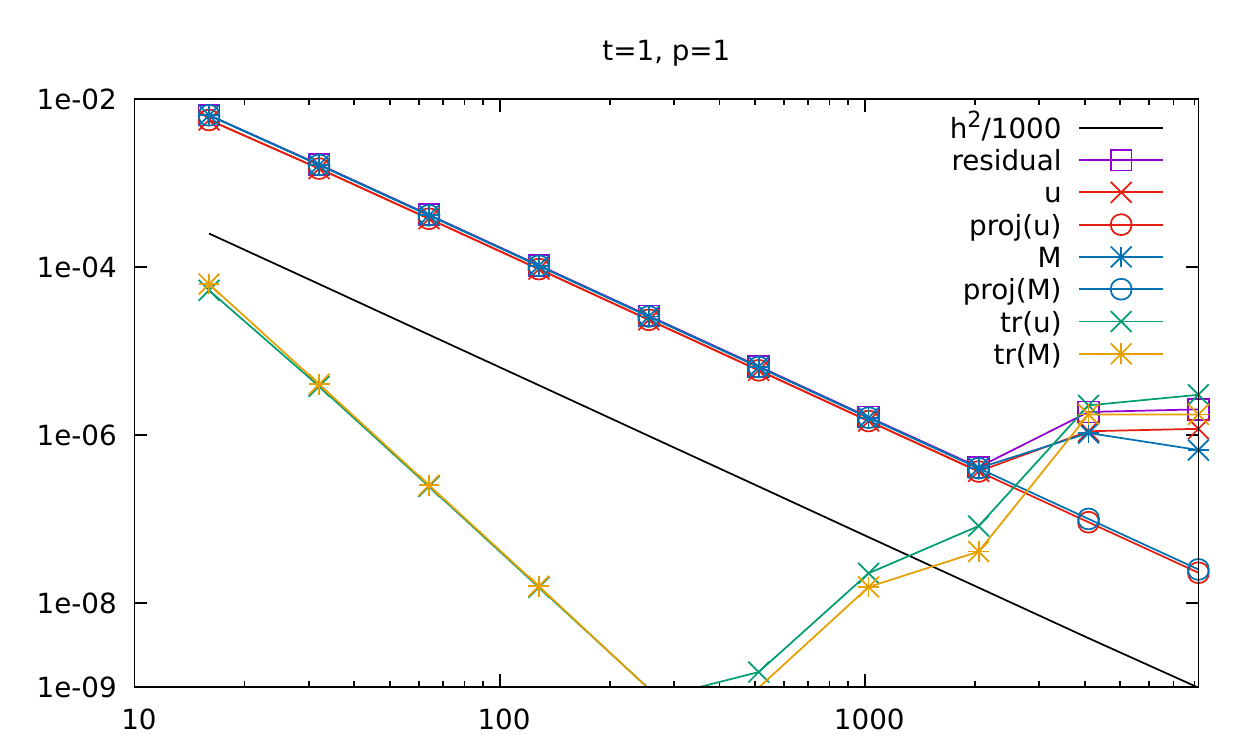}\hspace*{-0.5em}
\includegraphics[width=0.5\textwidth]{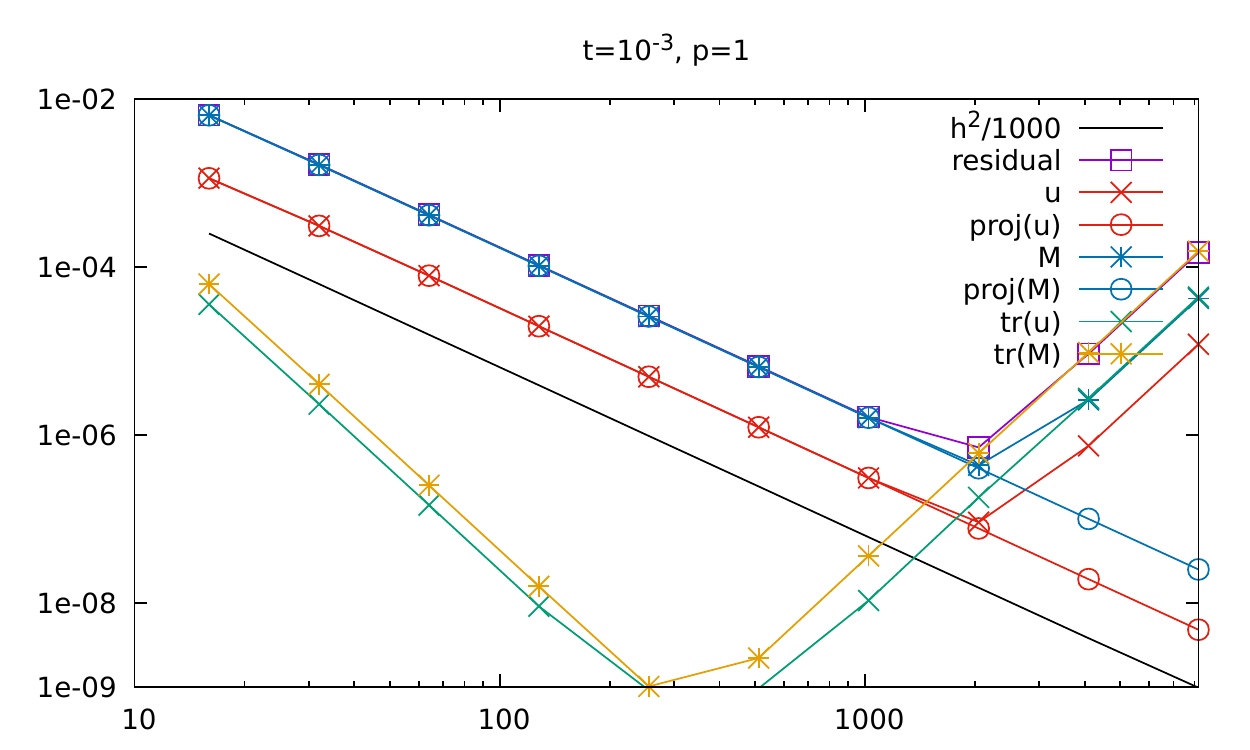}\\
\includegraphics[width=0.5\textwidth]{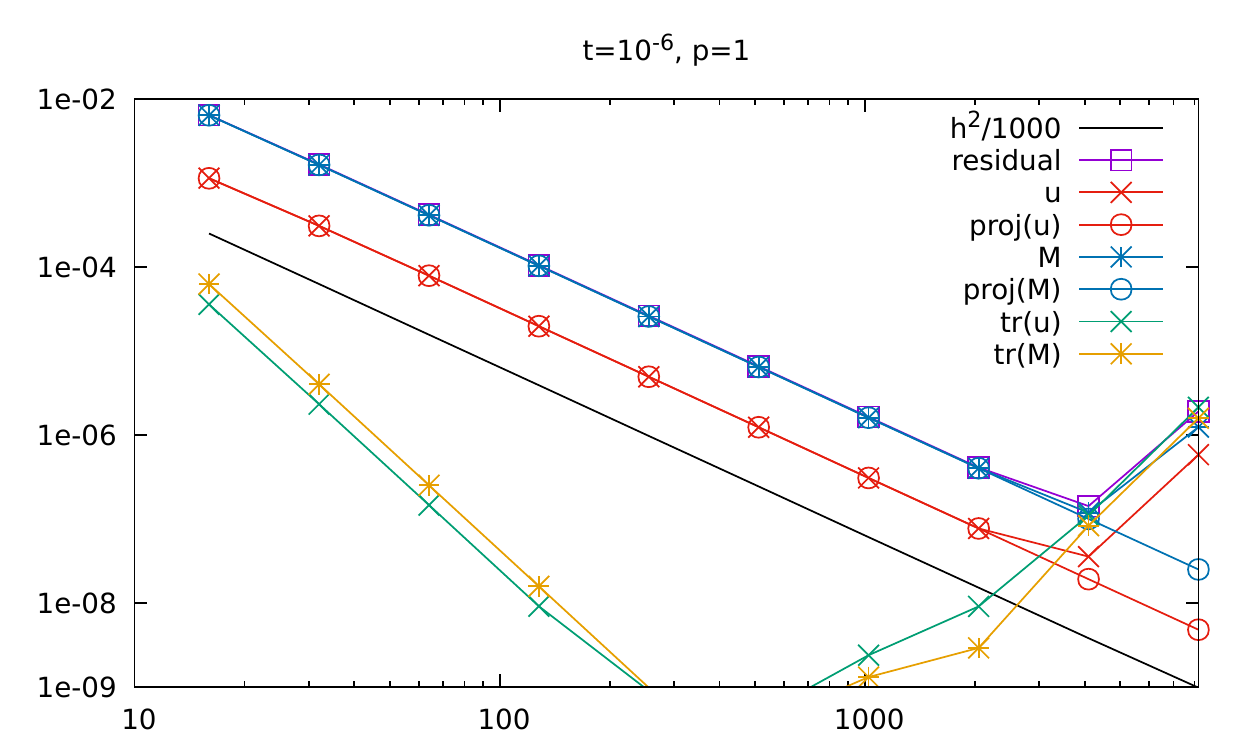}\hspace*{-0.5em}
\includegraphics[width=0.5\textwidth]{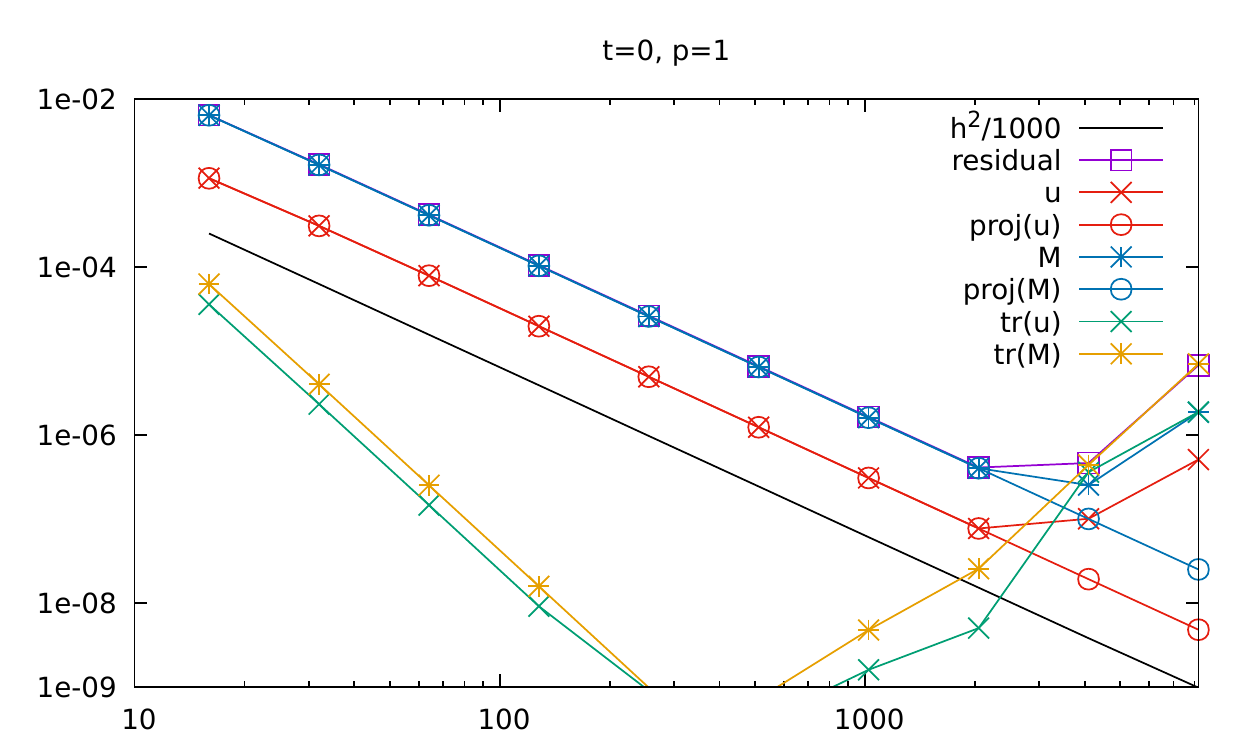}
\caption{Errors in the case of piecewise linear approximations, $t=1,10^{-3},10^{-6},0$.}
\label{fig_p1}
\end{figure}

\begin{figure}[htb]
\includegraphics[width=0.5\textwidth]{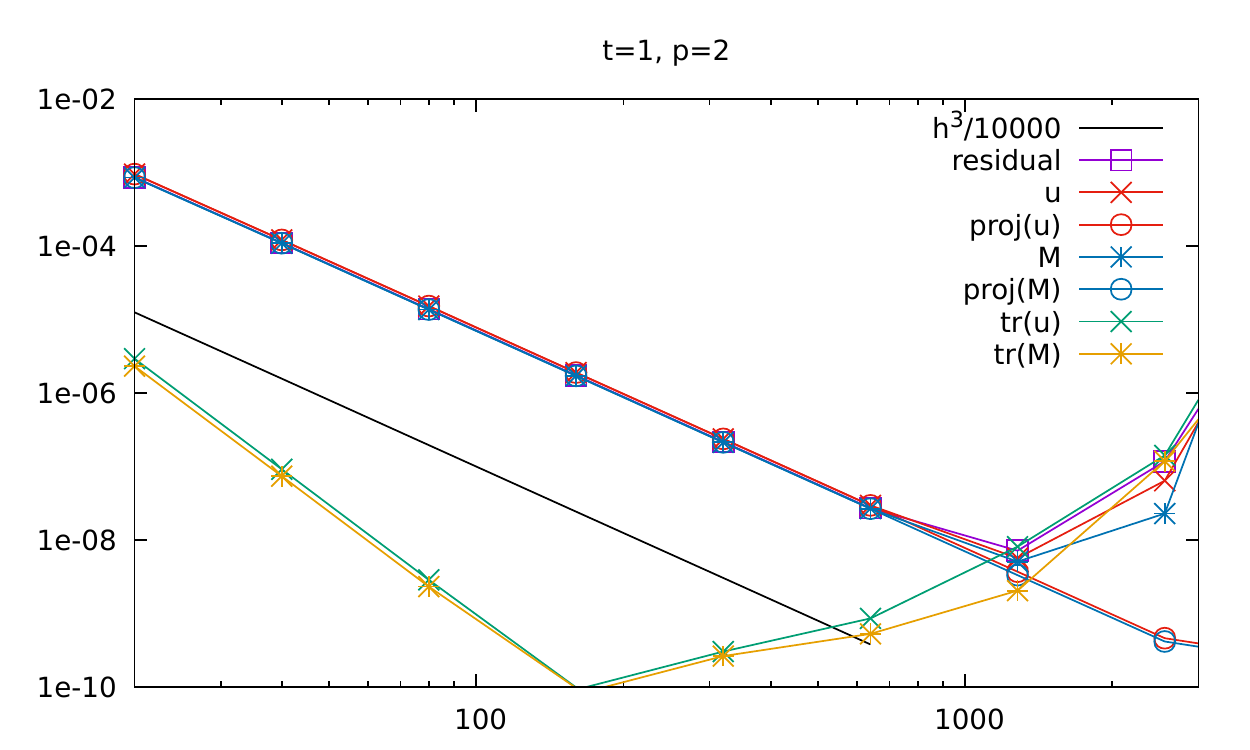}\hspace*{-0.5em}
\includegraphics[width=0.5\textwidth]{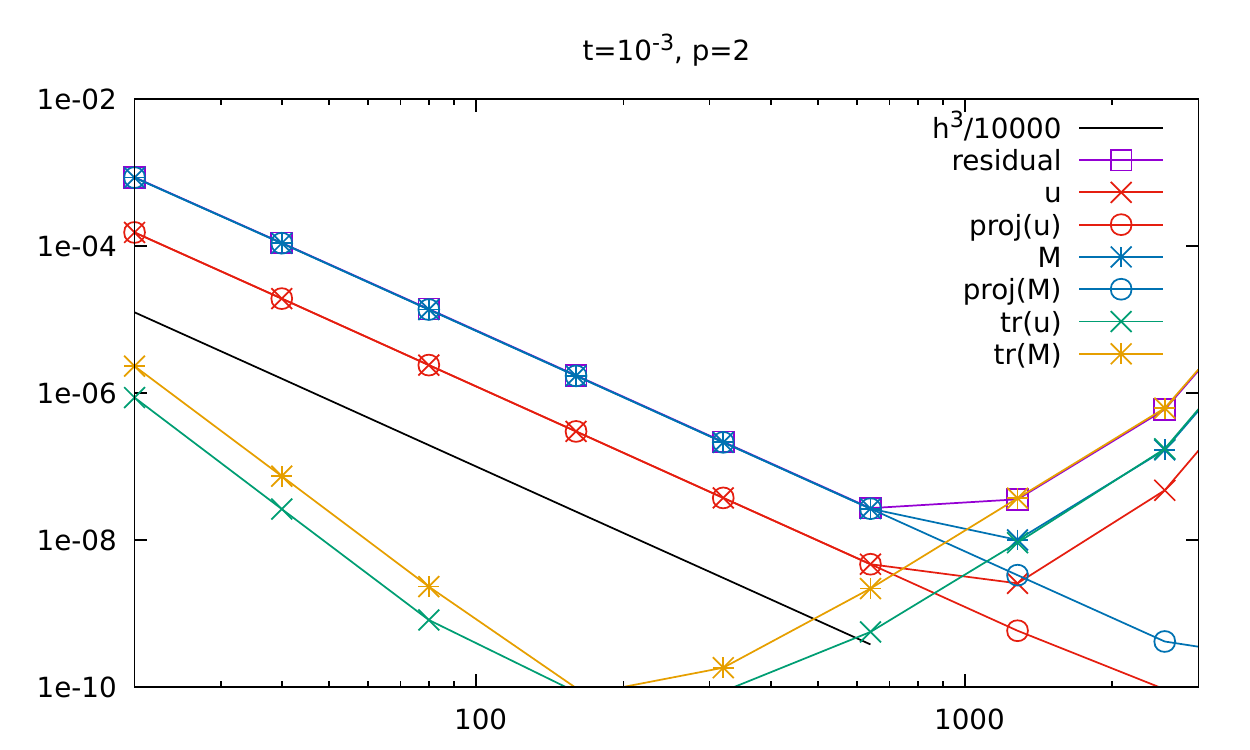}\\
\includegraphics[width=0.5\textwidth]{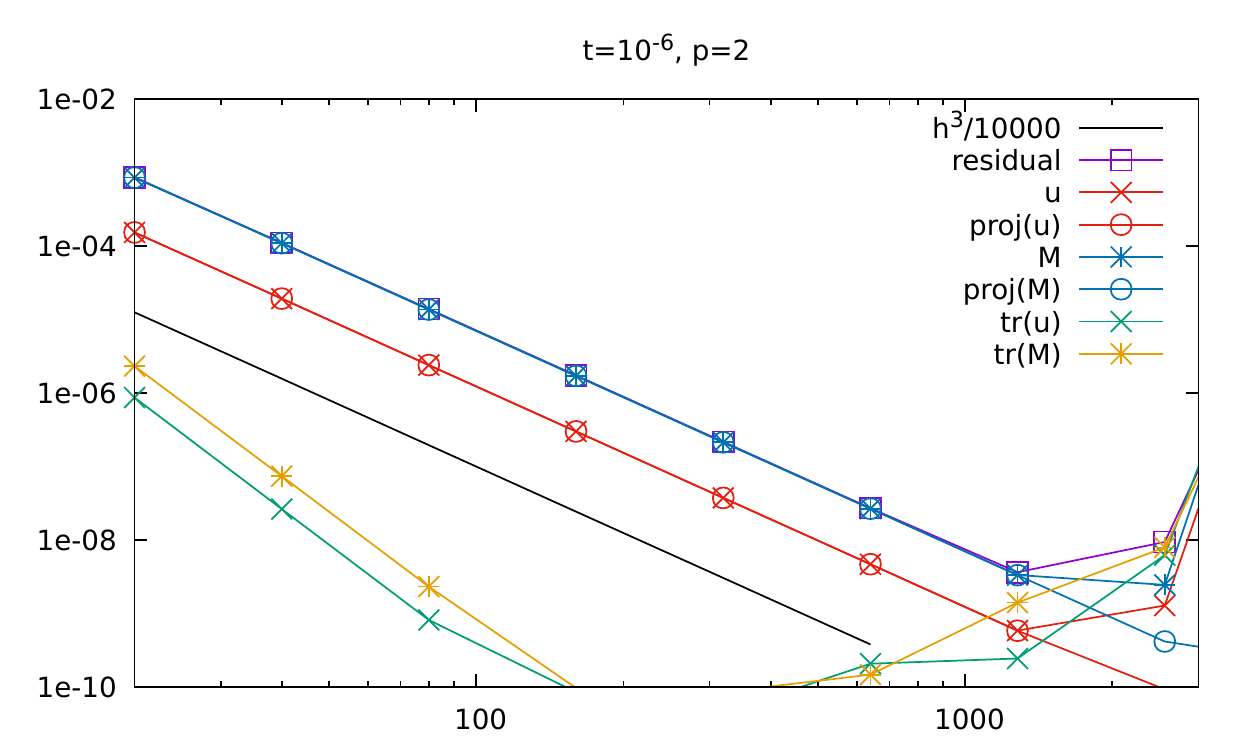}\hspace*{-0.5em}
\includegraphics[width=0.5\textwidth]{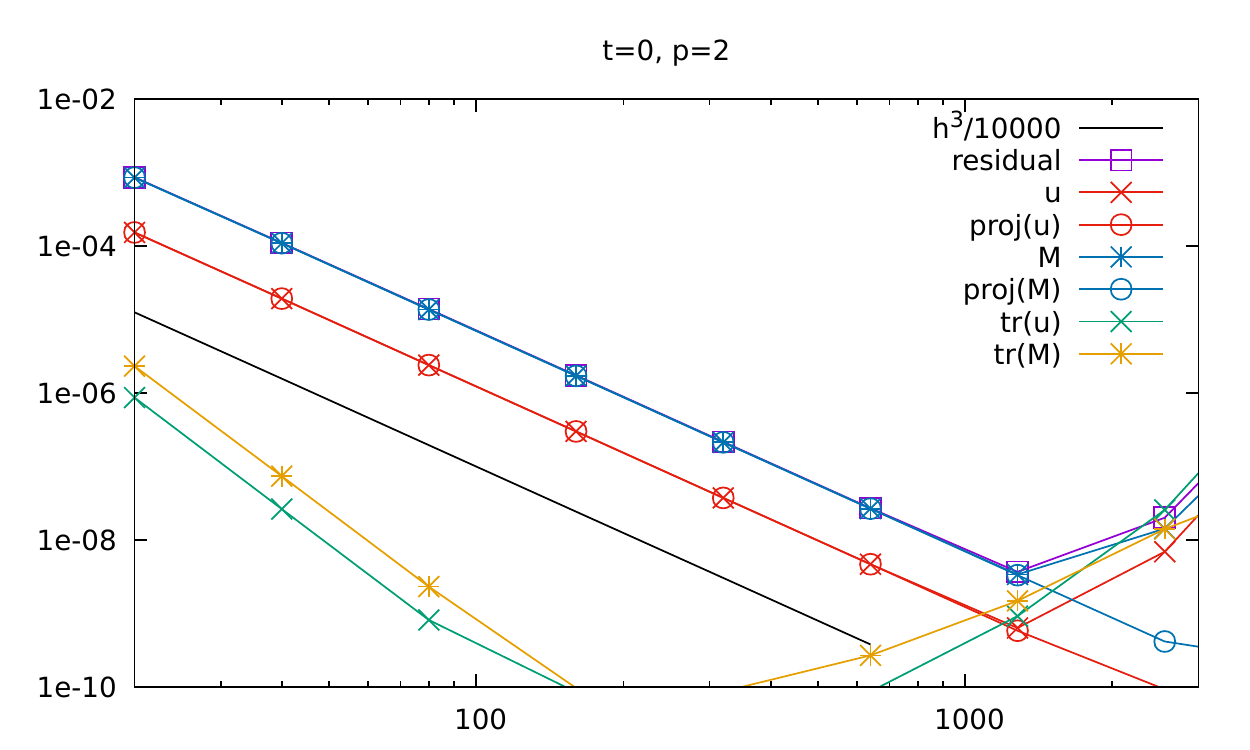}
\caption{Errors in the case of piecewise quadratic approximations, $t=1,10^{-3},10^{-6},0$.}
\label{fig_p2}
\end{figure}
%%%%%%%%%%%%%%%%%%%%%%%%%%%%%%%%%%%%%%%%%%%%%%%%%%%%%%%%%%%%%%%%%%%%%%%%%%%%%%%%%%%%%%%%

\bibliographystyle{siam}
\bibliography{/home/norbert/tex/bib/heuer,/home/norbert/tex/bib/bib}

\begin{thebibliography}{10}

\bibitem{Baccouch_14_LDG}
{\sc M.~Baccouch}, {\em The local discontinuous {G}alerkin method for the
  fourth-order {E}uler-{B}ernoulli partial differential equation in one space
  dimension. {P}art {I}: {S}uperconvergence error analysis}, J. Sci. Comput.,
  59 (2014), pp.~795--840.

\bibitem{CarstensenDG_16_BSF}
{\sc C.~Carstensen, L.~F. Demkowicz, and J.~Gopalakrishnan}, {\em Breaking
  spaces and forms for the {DPG} method and applications including {M}axwell
  equations}, Comput. Math. Appl., 72 (2016), pp.~494--522.

\bibitem{CelikerCS_06_LFO}
{\sc F.~Celiker, B.~Cockburn, and H.~K. Stolarski}, {\em Locking-free optimal
  discontinuous {G}alerkin methods for {T}imoshenko beams}, SIAM J. Numer.
  Anal., 44 (2006), pp.~2297--2325.

\bibitem{ChanHBTD_14_RDM}
{\sc J.~Chan, N.~Heuer, T.~Bui-Thanh, and L.~F. Demkowicz}, {\em Robust {DPG}
  method for convection-dominated diffusion problems {II}: {Adjoint} boundary
  conditions and mesh-dependent test norms}, Comput. Math. Appl., 67 (2014),
  pp.~771--795.

\bibitem{DemkowiczFHT_DAP}
{\sc L.~F. Demkowicz, T.~F\"uhrer, N.~Heuer, and X.~Tian}, {\em The double
  adaptivity paradigm ({How} to circumvent the discrete inf-sup conditions of
  {Babu\v{s}ka} and {Brezzi})}, {ICES Report} 19-07, The University of Texas at
  Austin, 2019.

\bibitem{DemkowiczG_11_ADM}
{\sc L.~F. Demkowicz and J.~Gopalakrishnan}, {\em Analysis of the {DPG} method
  for the {Poisson} problem}, SIAM J. Numer. Anal., 49 (2011), pp.~1788--1809.

\bibitem{DemkowiczH_13_RDM}
{\sc L.~F. Demkowicz and N.~Heuer}, {\em Robust {DPG} method for
  convection-dominated diffusion problems}, SIAM J. Numer. Anal., 51 (2013),
  pp.~2514--2537.

\bibitem{FuehrerHH_TOB}
{\sc T.~F{\"u}hrer, A.~Haberl, and N.~Heuer}, {\em Trace operators of the
  bi-{Laplacian} and applications}, IMA J. Numer. Anal.
\newblock (accepted for publication), arxiv:1904.07761.

\bibitem{FuehrerH_19_FDD}
{\sc T.~F{\"u}hrer and N.~Heuer}, {\em Fully discrete {DPG} methods for the
  {Kirchhoff}--{Love} plate bending model}, Comput. Methods Appl. Mech. Engrg.,
  343 (2019), pp.~550--571.

\bibitem{FuehrerHN_19_UFK}
{\sc T.~F{\"u}hrer, N.~Heuer, and A.~H. Niemi}, {\em An ultraweak formulation
  of the {Kirchhoff}--{Love} plate bending model and {DPG} approximation},
  Math. Comp., 88 (2019), pp.~1587--1619.

\bibitem{FuehrerHS_UFR}
{\sc T.~F{\"u}hrer, N.~Heuer, and F.-J. Sayas}, {\em An ultraweak formulation
  of the {Reissner}--{Mindlin} plate bending model and {DPG} approximation},
  Numer. Math.
\newblock (accepted for publication), arxiv:1906.04869.

\bibitem{HeuerK_17_RDM}
{\sc N.~Heuer and M.~Karkulik}, {\em A robust {DPG} method for singularly
  perturbed reaction-diffusion problems}, SIAM J. Numer. Anal., 55 (2017),
  pp.~1218--1242.

\bibitem{LepeMR_14_LFF}
{\sc F.~Lepe, D.~Mora, and R.~Rodr\'{\i}guez}, {\em Locking-free finite element
  method for a bending moment formulation of {T}imoshenko beams}, Comput. Math.
  Appl., 68 (2014), pp.~118--131.

\bibitem{Li_90_DTB}
{\sc L.~K. Li}, {\em Discretization of the {T}imoshenko beam problem by the
  {$p$} and the {$h$}-{$p$} versions of the finite element method}, Numer.
  Math., 57 (1990), pp.~413--420.

\bibitem{NiemiBD_11_DPG}
{\sc A.~H. Niemi, J.~A. Bramwell, and L.~F. Demkowicz}, {\em Discontinuous
  {P}etrov-{G}alerkin method with optimal test functions for thin-body problems
  in solid mechanics}, Comput. Methods Appl. Mech. Eng., 200 (2011),
  pp.~1291--1300.

\end{thebibliography}

\end{document}